\documentclass{article}
\usepackage{amssymb,amsthm,mathtools}
\usepackage[pdftex,bookmarks=true]{hyperref}
\usepackage{doi}

\DeclareMathOperator{\tr}{tr}

\theoremstyle{theorem}
\newtheorem{thm}{Theorem}
\newtheorem{lem}{Lemma}
\newtheorem{prop}{Proposition}

\theoremstyle{definition}
\newtheorem{defn}{Definition}

\begin{document}
\title{On the Combinatorial Core of Second-Order Quantum Argument Shifts in $U\mathfrak{gl}_d$}
\author{Yasushi Ikeda}
\date{September 30, 2025}
\maketitle

\begin{abstract}
We provide a complete, self-contained proof of Theorem~4 of~\cite{Ikeda2024} that reduces second-order generators of the quantum argument-shift algebra in the universal enveloping algebra $U\mathfrak{gl}_d$. We prove the necessary combinatorial identities---expressed as relations among polynomials with rational coefficients---by induction.
\end{abstract}

\section{Introduction}
Suppose that $\mathfrak g$ is a complex Lie algebra and let $\xi$ be an element of the dual space $\mathfrak g^*$. We write $\bar\partial_\xi$ for the constant vector field in the direction $\xi$ and the Poisson center of the symmetric algebra $S\mathfrak g$ is denoted by $\bar C$. Mishchenko and Fomenko~\cite{Mishchenko1978} showed that the algebra $\bar C_\xi$ generated by the set $\bigcup_{n=0}^\infty\bar\partial_\xi^n\bar C$ is Poisson-commutative.

Vinberg~\cite{Vinberg1991} asked whether this argument-shift algebra $\bar C_\xi$ admits a natural quantization in the universal enveloping algebra $U\mathfrak g$. This question has since been resolved in various settings by several authors (see, e.g.,~\cite{Nazarov1996,Tarasov2000,Rybnikov2006,Feigin2010,Futorny2015,Molev2019}). Generators of quantum argument-shift algebras are discussed in~\cite{Tarasov2000,Futorny2015,Molev2019,Chervov2006,Molev2013}.

Another approach is based on the notion of quantum partial derivatives $\partial^i_j$ on the universal enveloping algebras $U\mathfrak{gl}_d$~\cite{Gurevich2012}. We defined the quantum argument-shift operator $\partial_\xi=\tr(\xi\partial)$ and showed a quantum analogue of the Mishchenko and Fomenko theorem~\cite{IkedaSharygin2024}.

We studied iterated quantum argument shifts of central elements up to second order~\cite{Ikeda2022,Ikeda2024}. The simplification of the second-order shifts relies on Theorem~4 of \cite{Ikeda2024}. We provide a complete, self-contained proof of this theorem.

Section~\ref{sec:preliminaries} sets up notation and elementary identities for binomial coefficients, together with a simple criterion for polynomial identities. Section~\ref{sec:main} restates the identities in binomial (Theorem~\ref{thm:binomial}) and polynomial (Theorem~\ref{thm:polynomial}) form. Section~\ref{sec:proof} proves Theorems~\ref{thm:binomial} and~\ref{thm:polynomial} by reducing them to three intermediate propositions---Lemmas~\ref{lem:polynomial} and~\ref{lem:binomial}---each established by induction.

\section{Preliminaries}\label{sec:preliminaries}
Suppose that $x$ is an indeterminate.

\subsection{Binomial Coefficients and Basic Identities}
\begin{defn}
We define
\begin{align*}
\binom{x}0=1,&&\binom{x}n=\frac{x(x-1)\dotsm(x-n+1)}{n!}\in\mathbb Q[x]
\end{align*}
for any nonnegative integer $n$.
\end{defn}

Proposition~\ref{prop:minus} is elementary.

\begin{prop}\label{prop:minus}
We have
\begin{enumerate}
\item
$\displaystyle\binom{x}n=(-1)^n\binom{n-1-x}n$ and
\item
Pascal’s identity $\displaystyle\binom{x}{n+1}=\binom{x-1}{n+1}+\binom{x-1}n$
\end{enumerate}
for any nonnegative integer $n$.
\end{prop}

We introduce the convention used in the statement of Theorem~\ref{thm:polynomial}.

\begin{defn}
Suppose that $m$ and $n$ are complex numbers such that the complex number $n-m$ is a nonnegative integer. We define $\displaystyle\binom{n}m=\binom{n}{n-m}$.
\end{defn}

We have $\displaystyle\binom{-1}{-1}=\binom{-1}0=1$.

\subsection{Uniqueness from Values}
Proposition~\ref{prop:zero} is a useful criterion for proving polynomial identities.

\begin{prop}\label{prop:zero}
Suppose that $f(x)$ is a polynomial in the one indeterminate $x$ over an integral domain. We have $f(x)=0$ if and only if $\deg f(x)<\#f^{-1}(0)$.
\end{prop}

\begin{proof}
Suppose $f(x)\neq0$. We have $\#f^{-1}(0)\leq\deg f(x)$ since any integral domain embeds into the algebraic closure of its field of fractions.
\end{proof}

\section{Main Results}\label{sec:main}
Theorem~4 in \cite{Ikeda2024} is equivalent to Proposition~4.1 in the same paper. The second part of Proposition~4.1 is
\begin{multline}\label{eq:evenpolynomial}
x^{2n}\frac{(x+1)^m+(x-1)^m}2+\frac{(x+1)^{m+2n}+(x-1)^{m+2n}}2\\
=\sum_{k=0}^n\biggl(\binom{2n-k}k+\binom{2n-k-1}{k-1}\biggr)x^k\frac{(x+1)^{m+k}+(x-1)^{m+k}}2
\end{multline}
and
\begin{multline}\label{eq:oddpolynomial}
x^{2n+1}\frac{(x+1)^m+(x-1)^m}2+\frac{(x+1)^{m+2n+1}+(x-1)^{m+2n+1}}2\\
=\sum_{k=0}^n\binom{2n-k}k\biggl(x^{k+1}\frac{(x+1)^{m+k}+(x-1)^{m+k}}2\\
+x^k\frac{(x+1)^{m+k+1}+(x-1)^{m+k+1}}2\biggr)
\end{multline}
for any nonnegative integers $m$ and $n$ since
$$f^{(n)}_+(x)=\sum_{m=0}^{n+1}\frac{1+(-1)^{n-m}}2\binom{n}mx^m=\frac{(x+1)^n+(x-1)^n}2$$
for any nonnegative integer $n$. The equations~\eqref{eq:evenpolynomial} and~\eqref{eq:oddpolynomial} follow from
$$x^{2n}+(x+1)^{2n}=\sum_{m=0}^n\biggl(\binom{2n-m}m+\binom{2n-m-1}{m-1}\biggr)x^m(x+1)^m$$
and
$$x^{2n+1}+(x+1)^{2n+1}=\sum_{m=0}^n\binom{2n-m}m\bigl(x^{m+1}(x+1)^m+x^m(x+1)^{m+1}\bigr)$$
for any nonnegative integer $n$. Theorem~4 in~\cite{Ikeda2024} reduces to Theorems~\ref{thm:binomial} and~\ref{thm:polynomial}.

\begin{thm}\label{thm:binomial}
We have
\begin{multline*}
\binom{2n_1+n_2+2n_3+1}{2n_3}+\binom{n_2+2n_3}{2n_3}\\
=\sum_{n_4=0}^{n_3}\biggl(\binom{n_1+n_2+n_3+n_4+1}{2n_4}+\binom{n_1+n_2+n_3+n_4}{2n_4}\biggr)\\
\binom{n_1+n_3-n_4}{2(n_3-n_4)}
\end{multline*}
and
\begin{multline*}
\binom{2n_1+n_2+2n_3+2}{2n_3}+\binom{n_2+2n_3}{2n_3}\\
=\sum_{n_4=0}^{n_3}\binom{n_1+n_2+n_3+n_4+1}{2n_4}\\
\biggl(\binom{n_1+n_3-n_4+1}{2(n_3-n_4)}+\binom{n_1+n_3-n_4}{2(n_3-n_4)}\biggr)
\end{multline*}
for any nonnegative integers $(n_m)_{m=1}^3$.
\end{thm}

\begin{thm}\label{thm:polynomial}
We have
$$x^{2n}+(x+1)^{2n}=\sum_{m=0}^n\biggl(\binom{2n-m}m+\binom{2n-m-1}{m-1}\biggr)x^m(x+1)^m$$
and
$$x^{2n+1}+(x+1)^{2n+1}=\sum_{m=0}^n\binom{2n-m}m\bigl(x^{m+1}(x+1)^m+x^m(x+1)^{m+1}\bigr)$$
for any nonnegative integer $n$.
\end{thm}

One can verify Theorem~\ref{thm:polynomial} using the following Mathematica code.
\begin{verbatim}
In[1]:=
Simplify[
  x^(2 n) + (x + 1)^(2 n) -
    Sum[
      (Binomial[2 n - m, m] + Binomial[2 n - m - 1, m - 1])*
        x^m*(x + 1)^m,
      {m, 0, n}]]
Simplify[
  x^(2 n + 1) + (x + 1)^(2 n + 1) -
    Sum[
      Binomial[2 n - m, m]*
        (x^(m + 1)*(x + 1)^m + x^m*(x + 1)^(m + 1)),
      {m, 0, n}]]
Out[1]= 0
Out[2]= 0
\end{verbatim}

\section{Proof of the Main Results}\label{sec:proof}
We give proofs of Theorems~\ref{thm:binomial} and~\ref{thm:polynomial}. Suppose that $x$ and $y$ are commutative indeterminates.

\subsection{Proof of Theorem~\ref{thm:binomial}}
Theorem~\ref{thm:binomial} reduces to Lemma~\ref{lem:polynomial}.

\begin{lem}\label{lem:polynomial}
We have
\begin{multline}\label{eq:evenbinomial}
\binom{x+y+n}{2n}+\binom{x-y+n}{2n}\\
=\sum_{m=0}^n\binom{x+m}{2m}\biggl(\binom{y+n-m}{2(n-m)}+\binom{y-1+n-m}{2(n-m)}\biggr)
\end{multline}
and
\begin{multline*}
\binom{x+y+n}{2n+1}+\binom{x-y+n}{2n+1}\\
=\sum_{m=0}^n\binom{x+m}{2m+1}\biggl(\binom{y+n-m}{2(n-m)}+\binom{y-1+n-m}{2(n-m)}\biggr)
\end{multline*}
for any nonnegative integer $n$.
\end{lem}

One can verify the equation~\eqref{eq:evenbinomial} in Lemma~\ref{lem:polynomial} over a finite grid of parameter values using the following Mathematica code.
\begin{verbatim}
In[3]:=
lhs[x_, y_, n_] :=
  Binomial[x + y + n, 2 n] +
  Binomial[x - y + n, 2 n];
rhs[x_, y_, n_] :=
  Sum[
    Binomial[x + m, 2 m]*
      (Binomial[y + n - m, 2 (n - m)] +
       Binomial[y - 1 + n - m, 2 (n - m)]),
    {m, 0, n}];
AllTrue[
  Flatten @ Table[
    lhs[x, y, n] == rhs[x, y, n],
    {x, -20, 20}, {y, -20, 20}, {n, 0, 20}],
  TrueQ]
Out[5]= True
\end{verbatim}

\begin{proof}[Proof of Theorem~\ref{thm:binomial}]
We have
\begin{align*}
\MoveEqLeft\sum_{n_4=0}^{n_3}\biggl(\binom{n_1+n_2+n_3+n_4+1}{2n_4}+\binom{n_1+n_2+n_3+n_4}{2n_4}\biggr)\binom{n_1+n_3-n_4}{2(n_3-n_4)}\\
&=\binom{2n_1+n_2+2n_3+1}{2n_3}+\binom{-n_2-1}{2n_3}\\
&=\binom{2n_1+n_2+2n_3+1}{2n_3}+\binom{n_2+2n_3}{2n_3}
\end{align*}
and
\begin{multline*}
\sum_{n_4=0}^{n_3}\binom{n_1+n_2+n_3+n_4+1}{2n_4}\biggl(\binom{n_1+n_3-n_4+1}{2(n_3-n_4)}+\binom{n_1+n_3-n_4}{2(n_3-n_4)}\biggr)\\
=\binom{2n_1+n_2+2n_3+2}{2n_3}+\binom{n_2+2n_3}{2n_3}
\end{multline*}
by Proposition~\ref{prop:minus} and Lemma~\ref{lem:polynomial}.
\end{proof}

\begin{proof}[Proof of Lemma~\ref{lem:polynomial}]
We may assume that the indeterminate $x$ is a nonnegative integer by Proposition~\ref{prop:zero} since the polynomial algebra $\mathbb Q[y]$ is an integral domain. The proof is by induction on the nonnegative integer $x$. The theorem holds for $x=0$ by Proposition~\ref{prop:minus}. Suppose $x>0$. We have
\begin{align}\label{eq:oddbinomial}
\MoveEqLeft\binom{x+y+n}{2n+1}+\binom{x-y+n}{2n+1}\notag\\
&=\binom{x-1+y+n}{2n+1}+\binom{x-1-y+n}{2n+1}+\binom{x-1+y+n}{2n}+\binom{x-1-y+n}{2n}\notag\\
&=\sum_{m=0}^n\biggl(\binom{x-1+m}{2m+1}+\binom{x-1+m}{2m}\biggr)\biggl(\binom{y+n-m}{2(n-m)}+\binom{y-1+n-m}{2(n-m)}\biggr)\notag\\
&=\sum_{m=0}^n\binom{x+m}{2m+1}\biggl(\binom{y+n-m}{2(n-m)}+\binom{y-1+n-m}{2(n-m)}\biggr)
\end{align}
by Pascal's identity and the induction hypothesis. We prove the equation~\eqref{eq:evenbinomial}. We have
$$\binom{x+y}0+\binom{x-y}0=\binom{x}0\biggl(\binom{y}0+\binom{y-1}0\biggr)=2$$
and may assume $n>0$. We have
\begin{multline*}
\binom{x+y+n}{2n}+\binom{x-y+n}{2n}=\binom{x-1+y+n}{2n}+\binom{x-1-y+n}{2n}\\
+\binom{x+y+n-1}{2n-1}+\binom{x-y+n-1}{2n-1}
\end{multline*}
by Pascal's identity. We have
\begin{multline*}
\binom{x-1+y+n}{2n}+\binom{x-1-y+n}{2n}\\
=\sum_{m=0}^n\binom{x-1+m}{2m}\biggl(\binom{y+n-m}{2(n-m)}+\binom{y-1+n-m}{2(n-m)}\biggr)
\end{multline*}
by the induction hypothesis and
\begin{align*}
\MoveEqLeft\binom{x+y+n-1}{2n-1}+\binom{x-y+n-1}{2n-1}\\
&=\sum_{m=0}^{n-1}\binom{x+m}{2m+1}\biggl(\binom{y+n-1-m}{2(n-1-m)}+\binom{y-1+n-1-m}{2(n-1-m)}\biggr)\\
&=\sum_{m=1}^n\binom{x+m-1}{2m-1}\biggl(\binom{y+n-m}{2(n-m)}+\binom{y-1+n-m}{2(n-m)}\biggr)
\end{align*}
by the equation~\eqref{eq:oddbinomial}. We have the equation~\eqref{eq:evenbinomial} by Pascal's identity.
\end{proof}

\subsection{Proof of Theorem~\ref{thm:polynomial}}
We define
$$S_n(x)=\sum_{m=0}^n\binom{x-m}m\binom{m}{n-m}$$
for any integer $n$. Theorem~\ref{thm:polynomial} reduces to Lemma~\ref{lem:binomial}.

\begin{lem}\label{lem:binomial}
We have
\begin{enumerate}
\item
$\displaystyle\sum_{m=0}^n\binom{x-m}m\binom{y+m}{n-m}=S_n(x+y)$ and
\item
$\displaystyle\binom{x}n=S_n(x)+S_{n-1}(x-1)$
\end{enumerate}
for any nonnegative integer $n$.
\end{lem}

One can verify Lemma~\ref{lem:binomial} over a finite grid of parameter values using the following Mathematica code.
\begin{verbatim}
In[6]:=
S[n_, x_] := Sum[
  Binomial[x - m, m]*Binomial[m, n - m],
  {m, 0, n}];
AllTrue[
  Flatten @ Table[
    Sum[
      Binomial[x - m, m]*Binomial[y + m, n - m],
      {m, 0, n}] == S[n, x + y],
    {x, -20, 20}, {y, -20, 20}, {n, 0, 20}],
  TrueQ]
AllTrue[
  Flatten @ Table[
    Binomial[x, n] == S[n, x] + S[n - 1, x - 1],
    {x, -20, 20}, {n, 0, 20}],
  TrueQ]
Out[7]= True
Out[8]= True
\end{verbatim}

\begin{proof}[Proof of Theorem~\ref{thm:polynomial}]
We have
\begin{align*}
\MoveEqLeft\sum_{m=0}^n\biggl(\binom{2n-m}m+\binom{2n-m-1}{m-1}\biggr)x^m(x+1)^m\\
&=\sum_{m=0}^n\sum_{k=0}^m\biggl(\binom{2n-m}m+\binom{2n-m-1}{m-1}\biggr)\binom{m}kx^{m+k}\\
&=\sum_{m=0}^{2n}\sum_{k=0}^n\biggl(\binom{2n-k}k+\binom{2n-k-1}{k-1}\biggr)\binom{k}{m-k}x^m\\
&=2x^{2n}+\sum_{m=0}^{2n-1}\bigl(S_m(2n)+S_{m-1}(2n-1)\bigr)x^m\\
&=2x^{2n}+\sum_{m=0}^{2n-1}\binom{2n}mx^m=x^{2n}+(x+1)^{2n}
\end{align*}
and
\begin{align*}
\MoveEqLeft\sum_{m=0}^n\binom{2n-m}m\bigl(x^{m+1}(x+1)^m+x^m(x+1)^{m+1}\bigr)\\
&=(2x+1)\sum_{m=0}^n\binom{2n-m}mx^m(x+1)^m\\
&=(2x+1)\sum_{m=0}^{2n}\sum_{k=0}^n\binom{2n-k}k\binom{k}{m-k}x^m\\
&=(2x+1)\sum_{m=0}^{2n}S_m(2n)x^m\\
&=2x^{2n+1}+\sum_{m=0}^{2n}\bigl(2S_{m-1}(2n)+S_m(2n)\bigr)x^m\\
&=2x^{2n+1}+\sum_{m=0}^{2n}\bigl(S_m(2n+1)+S_{m-1}(2n)\bigr)x^m\\
&=2x^{2n+1}+\sum_{m=0}^{2n}\binom{2n+1}mx^m=x^{2n+1}+(x+1)^{2n+1}
\end{align*}
by Lemma~\ref{lem:binomial}.
\end{proof}

\begin{proof}[Proof of Lemma~\ref{lem:binomial}]
The proof is by induction on the nonnegative integer $n$.
\begin{enumerate}
\item
We have
$$\binom{x}0\binom{y}0=S_0(x+y)=1.$$
Suppose $n>0$. We may assume that the indeterminates $x$ and $y$ are nonnegative
integers by Proposition~\ref{prop:zero}. It is sufficient to show
\begin{equation}\label{eq:increment}
\sum_{m=0}^n\binom{x+1-m}m\binom{y+m}{n-m}=\sum_{m=0}^n\binom{x-m}m\binom{y+1+m}{n-m}.
\end{equation}
We have
\begin{multline}\label{eq:xincrement}
\sum_{m=0}^n\biggl(\binom{x+1-m}m-\binom{x-m}m\biggr)\binom{y+m}{n-m}\\
=\sum_{m=1}^n\binom{x-m}{m-1}\binom{y+m}{n-m}=S_{n-1}(x+y)
\end{multline}
and
\begin{multline}\label{eq:yincrement}
\sum_{m=0}^n\binom{x-m}m\biggl(\binom{y+1+m}{n-m}-\binom{y+m}{n-m}\biggr)\\
\sum_{m=0}^{n-1}\binom{x-m}m\binom{y+m}{n-1-m}=S_{n-1}(x+y)
\end{multline}
by Pascal's identity and the induction hypothesis. We obtain the equation~\eqref{eq:increment} by the equations~\eqref{eq:xincrement} and~\eqref{eq:yincrement}.
\item
We have
$$\binom{x}0=S_0(x)=1.$$
Suppose $n>0$. We may assume that the indeterminate $x$ is a nonnegative integer by Proposition~\ref{prop:zero}. It is sufficient to show
\begin{enumerate}
\item
$S_n(n)+S_{n-1}(n-1)=1$ and
\item
$\displaystyle S_n(x)-S_n(x-1)+S_{n-1}(x-1)-S_{n-1}(x-2)=\binom{x-1}{n-1}$.
\end{enumerate}
\begin{enumerate}
\item
We have
$$S_n(n)=\sum_{m=0}^n\binom{n-m}m\binom{m}{n-m}=\frac{1+(-1)^n}2$$
and
$$S_{n-1}(n-1)=\frac{1+(-1)^{n-1}}2.$$
\item
We have
$$S_n(x)-S_n(x-1)=\sum_{m=1}^n\binom{x-1-m}{m-1}\binom{m}{n-m}=S_{n-1}(x-1)$$
by the first part of this lemma and
$$S_{n-1}(x-1)-S_{n-1}(x-2)=S_{n-2}(x-2).$$
We have
\begin{multline*}
S_n(x)-S_n(x-1)+S_{n-1}(x-1)-S_{n-1}(x-2)\\
=S_{n-1}(x-1)+S_{n-2}(x-2)=\binom{x-1}{n-1}
\end{multline*}
by the induction hypothesis.\qedhere
\end{enumerate}
\end{enumerate}
\end{proof}

\end{document}